\documentclass[a4paper,12pt]{amsart}

\usepackage{fancybox}
\usepackage{amsmath,amssymb,amsthm}  	
\usepackage{color}
\usepackage{nicefrac}
\usepackage{mathrsfs}
\usepackage[margin=2.5cm]{geometry}

\usepackage{tikz-cd}
\usetikzlibrary{cd}

\newcommand{\drlabel}[1]{\label{#1}}

\def \C{\mathbb C}
\def \N{\mathbb N}

\newcommand*{\faktor}[2]{
  \raisebox{0.25\height}{\ensuremath{#1}}
  \mkern-3.5mu\diagup\mkern-2mu
  \raisebox{-0.25\height}{\ensuremath{#2}}
}

\theoremstyle{plain}
\newtheorem{satz}{Theorem}[section]

\newtheorem{lemma}[satz]{Lemma}
\newtheorem{corollary}[satz]{Corollary}
\newtheorem{proposition}[satz]{Proposition}

\theoremstyle{definition}
\newtheorem{beispiel}[satz]{Example}

\newtheorem{definition}[satz]{Definition}
\newtheorem{notiz}[satz]{Remark}

\usepackage{hyperref}		

\title{Equivalence of Curve Singularities and 
delta-Invariants}


\author{Reinhold H\"ubl}
\email{Reinhold.Huebl@dhbw.de}
\address{Zentrum f\"ur mathematisch-naturwissenschaftliches Basiswissen,
DHBW Mannheim, 68163 Mannheim, Germany}

\author{Irena Swanson}
\email{iswanso@purdue.edu}
\address{Department of Mathematics, Purdue University, 150 N. University Street,
West Lafayette, IN 47907-2067}

\date{\today}

\begin{document}
\allowdisplaybreaks
\belowdisplayskip=4pt
\abovedisplayskip=7pt

\begin{abstract}
We prove that if two parameterizations of a complete 
reduced noetherian curve over an algebraically closed 
field agree modulo a sufficiently 
large power of the maximal ideal,
then the two parameterizations are equivalent.
This strengthens some bounds from Greuel and 
Pfister \cite{GP}. In addition, we prove that if 
two reduced and irreducible curve singularities
are isomorphic modulo sufficiently high (and identical) 
powers of their respective maximal ideals,
then the completions of the two curves are isomorphic,
and the isomorphism of the full completions
agrees with the original isomorphism modulo some lower 
power of the maximal ideals.
We provide a new and better bound on the lower power,
strengthening the bound in Hironaka~\cite[Theorem B]{Hi} 
and Elias~\cite[Theorem 6]{Elias}.
\end{abstract}

\maketitle

\section{Introduction}\label{sect:intro}

Studying local curve singularities and their
classification has a long tradition, going back to
Zariski (cf.~\cite{Za}) and beyond.

A major breakthrough in understanding curve singularities
was obtained by Hironaka in~\cite{Hi} who showed that if
two curve singularities are isomorphic up to a
sufficiently high power of the maximal ideal
(depending on the delta--invariant of the singularity),
then the completions of the local rings are isomorphic
via an isomorphism which agrees with the original one up
to a suitable power of the maximal ideal. A different
approach was pursued by Greuel and Pfister~\cite{GP},
resp. Nguyen~\cite{Ng13}, who looked at parameterizations
$\varphi: k[[X_1, \ldots, X_n]] \longrightarrow R$ of
(complete) curve singularities and developed criteria,
depending on the delta--invariant as well, when two
parameterizations are equivalent (in a suitable sense
that will be explained in Definitions~\ref{def:equivdet}
and~\ref{def:equivdetno}),
depending on their behavior modulo high powers of the
maximal ideal.

Our motivation to approach these problems stems
from~\cite{HHMM} which established a close relation 
between the value
semigroup of a complete reduced and irreducible curve
singularity and its minimal generators.
These results and other basic facts will be covered in
Section~\ref{sect:facts}. We use them to improve for
reduced and analytically irreducible
curve singularities both the bounds of Greuel and Pfister
(in section~\ref{sect:determinacy})
and of Hironaka (in section~\ref{sect:Hironaka}).
We note that determinacy bounds in terms of the 
Milnor number, similar to the ones we obtain here, 
have been known for plane curves and some other types 
of singularities, cf.~\cite[example 5.6]{BGK},
~\cite[section 3.1, section 3.5]{GLS} 
and~\cite[Theorem 2.1]{Ng20}.

The research for this paper was done while the first
author was visiting Purdue University. He expresses his
gratitude to this institution for its support and
hospitality.

The authors are grateful to J. Elias and D. Kerner for helpful 
comments and suggestions.

\section{Definitions and known facts \label{sect:facts}}

Throughout $k$ is a field and $(R, \mathfrak m, k)$ is a
reduced, equicharacteristic noetherian curve singularity
with $k \subseteq R$ and with normalization $\overline{R}$.
We also assume that $\faktor{R}{\mathfrak m} = k 
= \faktor{\overline{R}}{\overline{\mathfrak m}}\,$ 
for every maximal ideal $\overline{\mathfrak m}$ of
$\overline R$ (which certainly is the case if $k$
is algebraically closed).

Except for the last part of Section~\ref{sect:determinacy},
we restrict to unibranch curve singularities, i.e.,
we assume that $\overline{R}$ is local as well, hence
it is a regular local ring of dimension~$1$ with a unique
maximal ideal $\overline{\mathfrak m}$.
Then $R$ is an analytically irreducible domain.

A crucial tool to understand curve singularities is the
singularity degree or delta--invariant.

\begin{definition} (\cite{GP})
When $\faktor{R}{\mathfrak m}$ is a finite-dimensional 
$k$-vector
space and $\overline{R}/R$ is module-finite, then the
\textbf{delta-invariant} of $R$ (with respect to $k$) is
	$$ \delta_k(R) = \mathrm{dim}_k\left(
    \faktor{\overline{R}}{R}\right) $$
and the \textbf{conductor degree} of $R$ (with
respect to $k$) is
	$$ c_k(R) = \mathrm{dim}_k\left(
    \faktor{\overline{R}}{\mathfrak C_R}\right), $$
where $\mathfrak C_R = \mathrm{Ann}_R(\overline{R}/R)$
is the conductor of~$R$, and it is an ideal both in
$R$ and in $\overline{R}$. As $k$ is fixed in our context,
we may drop the subscript~$\phantom{}_k$.
\end{definition}

These invariants and their relations to other invariants
of curve singularities have been studied extensively over the last decades,
c.f.~\cite{Hi},~\cite{Te}, \cite{Ca}, \cite{Ku} or~\cite{GP},
and they hold much numerical information about the singularity.

The unibranch curve singularity assumption implies that
$\delta(R) = \delta_k(R) = \delta_k(\widehat{R}) 
= \delta(\widehat{R})$
and that
$c(R) = c_k(R) = c_k(\widehat{R}) = c(\widehat{R})$.
Since these invariants of interest pass to completions,
we may assume that $R$ is a complete local domain. Then
$\overline{R} = k[[t]]$ for some regular parameter~$t$
of~$\overline{R}$, and $R \subseteq \overline{R}$ may 
be described by an (irreducible) parameterization
	$$ \varphi : P = k[[X_1, \ldots, X_N]] 
    \longrightarrow \overline{R}, $$
i.e., $R = k[[x_1, \ldots, x_n]] \subseteq \overline{R}$ 
with $x_i = \varphi(X_i) \in \overline{\mathfrak m}$.
Let $v$ be the $t$-adic valuation on $\overline{R}$, let
	$$ V = V(R) = \{ v(r) \, \vert \,\, 
    r \in R \setminus \{ 0 \} \} $$
be the value semigroup of $R$, and let $g=g(V)$ be
its genus and $c=c(V)$ its conductor.
For a fractional $R$-ideal $I$ we set
	$$ V(I) = \{ v(r) \, \vert \,\, r \in I 
    \setminus \{ 0 \} \}. $$

\begin{lemma}\drlabel{lem:delta-genus}
	$\delta(R) = g(V), \quad c(R) = c(V)$.
\end{lemma}

\begin{proof}
Viewing $\overline{R}$ as a fractional $R$-ideal we obtain
from~\cite[2.9]{HK} that
	$$ \delta(R) = \mathrm{dim}_k\left(
    \faktor{\overline{R}}{R}\right)
	= \vert V(\overline{R}) \setminus V(R) \vert = g(V). $$
By~\cite[2.4]{HK},
	$\mathfrak C_R = t^{c(V)} \cdot \overline{R}$,
so the argument for the conductor degree is similar.
\end{proof}

\begin{notiz}\drlabel{rem:genus-delta}
From Lemma~\ref{lem:delta-genus} we conclude immediately
that
	$$ \delta(R) \leq c(R) \leq 2 \cdot \delta(R) $$
as $\frac {1}{2} \cdot c(V) \leq g(V) \leq c(V)$
(cf.~\cite[Lemma 2.14]{RGS}).
\end{notiz}

\begin{definition}
For an $\mathfrak m$--primary ideal $I$ in a noetherian
local ring $(R,\mathfrak m,k)$, the {\bf Hilbert 
function of $I$} is
	$$ H_I(n) = \mathrm{dim}_k 
    \left(\faktor{I^n}{I^{n+1}}\right). $$
When $R$ is one-dimensional,
then $H_I(n)$ is constant for all sufficiently large $n$,
and in fact it equals the {\bf multiplicity} $e(I; R)$ 
of $I$ for all $n \geq r(I)$, the {\bf reduction number} of~$I$. 
When $I = \mathfrak m$, we also write $H_R$ 
for $H_I$ and $e(R)$ for $e(\mathfrak m; R)$ (the 
multiplicity of $R$).
\end{definition}

\begin{notiz}\drlabel{rem:ered}
Let $R$ be a noetherian local ring,
$I$ a proper ideal in $R$
and $x \in I$ a non-zerodivisor.
By Nakayama's Lemma, $(x)$ is a 
reduction of $I$ with reduction number at most $r$
if and only if the $x$-multiplication
	$$ \mu_x : \faktor{I^n}{I^{n+1}} \longrightarrow 
    \faktor{I^{n+1}}{I^{n+2}} $$
is an isomorphism for all $n \ge r$,
equivalently, for $n = r$.
\end{notiz}

The main result of~\cite{HHMM} is the following.

\begin{satz}\drlabel{thm:exist-unique} {\rm (\cite{HHMM})}
In the above situation there exist unique integers 
$1 \leq a_1 < a_2 < \cdots < a_n$ with the 
following properties:
\begin{enumerate}
\item $n = \mathrm{edim}(R)$.
\item If $x_1, \ldots, x_n \in R$ are elements satisfying 
$v(x_i) = a_i$ for all $i$, then
	$R = k[[x_1, \ldots, x_n]]$.
\end{enumerate}
The set $\{a_1, \ldots, a_n\}$ equals
	$V(\mathfrak m) \setminus V(\mathfrak m^2)$.
\end{satz}

\begin{definition}
The sequence $a_1 < a_2 < \cdots < a_n$ as in the last
theorem is called \textbf{Herzog--Kunz sequence} of $R$,
and elements $x_1, \ldots, x_n \in R$ with $v(x_i) = a_i$
are called \textbf{Herzog--Kunz generators} of $R$.
A parameterization
	$$ \varphi:P = k[[X_1,\ldots, X_n]] 
    \longrightarrow \overline{R} $$
with $v(\varphi(X_i)) = a_i$ is called a
\textbf{Herzog--Kunz parameterization} of~$R$.
\end{definition}

In~\cite{GP} Greuel and Pfister considered more generally
parameterizations
    $$ \varpi: k [[X_1, \ldots, X_N]] 
    \longrightarrow \widetilde{R} $$
of curve singularities with $r$ branches, where
$\widetilde{R} = k[[t_1]] \times \cdots \times k[[t_r]]$
(cf.~\cite[Definition 7]{GP}), and they call such a
parameterization {\bf primitive} if $\widetilde{R} =
\overline{R}$ is the normalization of $R =
\operatorname{im}(\varphi)$. Furthermore, they defined:

\begin{definition}\drlabel{def:equivdet}
Let $ \varphi, \psi : P = k[[X_1, \ldots, X_N]]
\longrightarrow \widetilde{R}$ be two parameterizations
(of possibly different curves).
Then $\varphi$ and $\psi$
are called \textbf{right-left-equivalent} or
$\mathbf{\mathcal{A}\text{-equivalent}}$
($\varphi \sim_{\mathcal{A}} \psi$) if there
exist $\sigma \in \mathrm{Aut}_k(P)$ and
$\tau \in \mathrm{Aut}_k(\widetilde{R})$ such
that $\tau \circ \varphi \circ \sigma^{-1} = \psi$.

It $\tau = \mathrm{id}_{\widetilde{R}}$, then $\varphi$
and $\psi$ are called \textbf{left-equivalent} or
$\mathbf{\mathcal{L}\text{-equivalent}}$
($\varphi \sim_{\mathcal{L}} \psi$).

\end{definition}

\begin{definition}\drlabel{def:equivdetno}
A parameterization $\varphi : P = k[[X_1, \ldots, X_N]]
\longrightarrow \widetilde{R}$ is called
$j$-$\mathcal{A}$-\textbf{de-termined}
(resp. $j$-$\mathcal{L}$-\textbf{determined})
if a parameterization $\psi : P = k[[X_1, \ldots, X_N]]
\longrightarrow \widetilde{R}$
with $\varphi = \psi \pmod{\overline{\mathfrak m}^{j+1}}$
is $\mathbf{\mathcal{A}\text{-equivalent}}$ to $\varphi$
(resp. $\mathbf{\mathcal{L}\text{-equivalent}}$ 
to $\varphi$).
\end{definition}

\section{Bounds on determinacy in delta\label{sect:determinacy}}

In this section we improve on some determinacy bounds from
Greuel and Pfister~\cite{GP}, focusing on showing that if
two primitive parameterizations of a curve are the same
modulo a sufficiently high power of the maximal ideal,
then the two curves are actually isomorphic.

Theorem~\ref{thm:det-cond}
is a strengthening of \cite[Proposition 11]{GP},
and Corollary~\ref{cor:det-delta}
is a strengthening of \cite[Proposition 10]{GP}
for unibranch curve singularities.
The end of the section contains also a strengthening
of \cite[Proposition 11]{GP} for multi-branched curves.
We present examples where our bounds are sharp.

Greuel and Pfister proved in \cite[Proposition 11]{GP}
that any primitive parameterization $\varphi$ of a curve
singularity is $(2 c(R) - 1)$-determined. As $a_n+1 \leq
c(R) +e(R) \leq 2 \cdot c(R)$, where $e(R) =
e(\mathfrak m; R)$ denotes the multiplicity of $R$,
the following is a strengthening for unibranch curves.

\begin{satz}\drlabel{thm:det-cond}
Let $R$ be a unibranch complete curve with conductor
degree $c(R)$ and with Herzog-Kunz sequence
$a_1 < a_2 < \cdots < a_n$ and set
	$$ d_R = \mathrm{max}\{ c(R), a_n+1 \}. $$
Then any primitive parameterization $\varphi :
P = k[[X_1, \ldots, X_N]]
\longrightarrow \overline{R} = k[[t]]$ of $R$ is
$(d_R{-}1)$-$\mathcal{L}$-determined (hence
$(d_R{-}1)$-$\mathcal{A}$-determined).
\end{satz}

\begin{proof}
Let $\psi : P \longrightarrow \overline{R}$ be another
parameterization of a curve with
	$\varphi = \psi \pmod{\overline{\mathfrak m}^{d_R}}$.
Write $\varphi(X_i) = \psi(X_i) + h_i$
for some $h_i \in t^{d_R} \cdot \overline{R}$. As
	$v(h_i) \geq d_R \geq c(R)$,
we have that $h_i \in R$,
hence in particular $\psi(X_i) \in R$. As
	$v(h_i) \geq d_R > a_n$,
we deduce from~\cite[(3.1)]{HHMM} that
	$$ h_i \in \mathfrak m^2 , \qquad
	h_i = \sum\limits_{j=1}^N g_{i,j} \cdot \varphi(X_j)
    \quad \text{ for some }\, g_{i,j}\in \mathfrak m. $$
Let $G_{i,j} \in (X_1, \ldots, X_N) \cdot P$ be such that
	$g_{i,j} = G_{i,j}(\varphi(X_1), \ldots, \varphi(X_N))$.
Thus, setting $H_i = \sum\limits_{j=1}^n G_{i,j}
\cdot X_j$ we have
$h_i = H_i(\varphi(X_1), \ldots, \varphi(X_N))$,
hence
	$$ \varphi(X_i) - H_i(\varphi(X_1), \ldots,
    \varphi(X_N)) = \psi(X_i) $$
and therefore
	$$ \sigma : P \longrightarrow P, \quad X_i \mapsto X_i
    - H_i(X_1, \ldots, X_N) $$
is an automorphism of $P$ (as $H_i \in (X_1, \ldots,
X_N)^2 \cdot P$) with $\varphi \circ \sigma = \psi$,
proving the proposition.
\end{proof}

It is well-known that $a_1 = e(R)$ and that 
$a_1 - 1 \le \delta(R)$.

We thank Michael Hellus for the proof of the 
following lemma.

\begin{lemma}\drlabel{lem:an-delta}
Let $a_1 < \cdots < a_n$ be the Herzog--Kunz sequence 
of $R$. Then
	$$ a_n \leq 2 \cdot \delta(R) + 1. $$
\end{lemma}

\begin{proof}
Let $V$ 
be the value semigroup of $R$ with minimal generators
$b_1, \ldots, b_N$. Then
	$\{a_1, \ldots, a_n \} \subseteq \{ b_1, \ldots, b_N\}$
by~\cite[(2.5)]{HHMM}.5). Thus by
Lemma~\ref{lem:delta-genus} it suffices to show
	$b_N \leq 2 \cdot g(V) + 1$.

With $V_+ = V \setminus \{0\}$ we have
	$\{ b_1, \ldots, b_N \} = V_+ \setminus (V_+ + V_+)$,
hence it suffices to show that any $\lambda \in\N$ with
$\lambda \geq 2 \cdot g(V) + 2$ is contained in
$V_+ + V_+$. By Remark~\ref{rem:genus-delta},
$2 \cdot g(V) + 2 \geq c(V)$,
hence certainly $\lambda \in V_+$.

Let $t = \lfloor \frac {\lambda}{2} \rfloor$. Then $t \geq
g(V) + 1$, and we have $t$ pairwise distinct presentations
	$$ \begin{array} {l c l c l}
	\lambda & = & 1 & + & (\lambda - 1), \\
	\lambda & = & 2 & + & (\lambda - 2), \\
	& \vdots & & & \\
	\lambda & = & t & + & (\lambda - t).\\
	\end{array} $$
As $t > g(V)$ for at least one equation
	$\lambda = i + (\lambda - i)$,
neither $i$ nor $\lambda-i$ can be gaps of $V$, i.e.,
both $i$ and $\lambda - i$ are in $V_+$, proving the lemma.
\end{proof}

The following strengthens~\cite[Corollary 10]{GP} 
for unibranch curve singularities.

\begin{corollary}\drlabel{cor:det-delta}
Let $ \varphi : P=k[[X_1, \ldots, X_N]] \longrightarrow
\overline{R} $ be a parameterization of a
unibranch (complete) curve $R$ with
delta--invariant $\delta(R)$. Then $\varphi$ is
$(2 \cdot \delta(R) + 1)$-$\mathcal{L}$-determined
(hence $(2 \cdot \delta(R){+}1)$-$\mathcal{A}$-determined).
\end{corollary}

\begin{proof}
By Remark~\ref{rem:genus-delta} we have $c(R) \leq 2
\cdot \delta(R)$, from Lemma~\ref{lem:an-delta} we deduce
that $a_n+1 \leq 2 \cdot \delta(R) + 2$, and then we apply
Theorem~\ref{thm:det-cond}.
\end{proof}

\begin{beispiel}\drlabel{ex:cor-sharp}
The bound in Corollary~\ref{cor:det-delta} is sharp in
some cases when the bound of~\cite[Proposition 11]{GP} is
not sharp, as in the following example of the
parameterization
$\varphi: k[[X_1, \ldots, X_5]] \longrightarrow k[[t]]$
given by
	$$ \varphi(X_1) = t^{5},\, \varphi(X_2) = t^{11},
    \, \varphi(X_3) = t^{17},
	\, \varphi(X_4) = t^{23}, \, \varphi(X_5) = t^{29}, $$
i.e.,
we consider the curve singularity
$R = k[[t^5, t^{11}, t^{17}, t^{23}, t^{29}]]$, which is a
numerical semigroup ring with maximal embedding dimension
and with Herzog--Kunz sequence $\{5, 11, 17, 23, 29\}$.
Its value semigroup is $V = \langle 5, 11, 17, 23, 29
\rangle$, its delta-invariant is
	$\delta(R) = g(V) = 14$,
and its conductor degree is~$c(R) = 25$.
Hence
	$a_5 = 29 = 2 \cdot \delta(R) + 1 > c(R)$
and $\varphi$ is $29$-$\mathcal{L}$-determined.
Since $2 \cdot c(R)-1 = 49 > 29$,
the bound in~\cite[Proposition 11]{GP} is not sharp.
We verify that $29$ is indeed sharp,
i.e., that $\varphi$ is not $28$-$\mathcal{L}$-determined.
Namely, for the parameterization
$\psi: k[[X_1, \ldots, X_5]] \longrightarrow k[[t]]$
given by
	$$ \psi(X_1) = t^{5}, \, \psi(X_2) = t^{11},
    \, \psi(X_3) = t^{17},
	\, \psi(X_4) = t^{23}, \, \psi(X_5) = t^{30} $$
we have that
	 $\varphi = \psi \pmod{\overline{\mathfrak m}^{29}}$,
but there is no $\sigma \in \mathrm{Aut}_k(P)$ with
$\psi = \varphi \circ \sigma^{-1}$, as
$\mathrm{im}(\psi) = k[[t^5, t^{11}, t^{17}, t^{23}]]
\subsetneq R.$
\end{beispiel}

\begin{beispiel}
The bounds in Theorem~\ref{thm:det-cond},
in Corollary~\ref{cor:det-delta},
and in
Greuel and Pfister~\cite[Proposition 11]{GP}
can be sharp simultaneously,
as is the case with the parameterization
$\varphi: k[[X_1, \ldots, X_n]] \longrightarrow k[[t]]$
with $\varphi(X_i)=t^{n-1+i}$
of~\cite[Rem. 13]{GP}. Here
	$c(R) = n$,
	$a_n = 2n-1 = 2 \cdot c(R) - 1$,
	$\delta(R) = n-1$,
hence $d_R = 2n = 2 c(R) = 2 \delta(R) + 2$.
Thus $\varphi$ is
$(2n-1)$-$\mathcal{L}$-determined but obviously
not $(2n-k)$-$\mathcal{L}$-determined for any 
$k\geq 2$ as $\varphi(X_n) = 0 
\pmod{\overline{\mathfrak m}^{2n-k+1}}$.

\smallskip
In this context we should point out a misprint
in~\cite[Corollary 12]{GP}. The parameterization
$\varphi: k[[X_1, X_2]] \longrightarrow k[[t]]$ with
$\varphi(X_1)=t^{2}$, $\varphi(X_2) = t^3$ satisfies
$c(R) = 2$ and $\delta(R) = 1$, and $\varphi$ is
$(2 \cdot c(R)-1)= 3$-$\mathcal{L}$-determined
but obviously not
$(4 \cdot \delta(R)-2) = 2$-$\mathcal{L}$-determined.
The bound~\cite[Proposition 11]{GP} combined with
$c(R) \leq 2 \cdot \delta(R)$ only gives the bound
$4 \cdot \delta(R) - 1$.
\end{beispiel}

For multi-branched curves, a value-optimal presentation as
in the unibranched case is not available. Nevertheless,
some partial results can be obtained which recover the
bound of~\cite[Proposition 11]{GP}.
Assume for the rest of this section
that $R$ has $\rho \geq 2$ branches and let
$\mathrm{Min}(R) =
\{ \mathfrak p_1, \ldots, \mathfrak p_{\rho} \}$.
Set $R_i =\faktor{R}{\mathfrak p_i}$ and let
$\overline{R_i}$ be its normalization. Again
we assume that $\faktor{R}{\mathfrak m} = k =
\faktor{\overline{R_i}}{\overline{\mathfrak m_i}}$
for all $i$.
Set
	$$ \widetilde{R} = \prod\limits_{i=1}^{\rho} R_i, \quad
	\overline{R} = \prod\limits_{i=1}^{\rho}
    \overline{R_i} =\prod\limits_{i=1}^{\rho} k[[t_i]]. $$
Then $\overline{R}$ is the normalization of $R$ and the
inclusion $R\subseteq \overline{R}$ factors as
	$$ R \xrightarrow[\phantom{leer}]{\varepsilon} 
    \widetilde{R}
    \xrightarrow[\phantom{leer}]{\iota} \overline{R}. $$
For $i = 1, \ldots, \rho$ set $\mathfrak q_i =
\bigcap\limits_{j\neq i} \mathfrak p_j$.

\begin{lemma}
In the above situation
\begin{itemize}
\item[i)] The canonical map $R \longrightarrow R_i$
induces an inclusion $\mathfrak q_i \subseteq R_i$.
\item[ii)] The conductor of $\widetilde{R}/R$ is
$\mathfrak C_{\widetilde{R}/R} =
(\mathfrak q_1, \ldots, \mathfrak q_{\rho})
\subseteq \widetilde{R}$.
\end{itemize}
\end{lemma}

\begin{proof}
i) If $q \in \mathfrak q_i$ and $q = 0
\pmod{\mathfrak p_i}$, then
	$q \in \bigcap\limits_{j = 1}^{\rho} \mathfrak p_j 
    = (0)$.

ii) If $q_i \in \mathfrak q_i$ and $r_i \in R_i$, then
$q_i \cdot r_i$ is a well-defined element of $R$,
independent of the representative (as $q_i \cdot p_i = 0$
for all $p_i \in \mathfrak p_i$). Thus, for any
$(r_1, \ldots, r_{\rho}) \in \widetilde{R}$ and any
$(q_1, \ldots, q_{\rho}) \in (\mathfrak q_1, \ldots,
\mathfrak q_{\rho})$, we obtain
$r = q_1 \cdot r_1 + \cdots + q_{\rho} \cdot r_{\rho}
\in R$ with
	$$ \varepsilon(r) = (q_1, \ldots, q_{\rho})
    \cdot (r_1, \ldots, r_{\rho}). $$
Conversely, if $c = (c_1, \ldots, c_{\rho}) \in
\mathfrak C_{\widetilde{R}/R}$, then
    $$c \cdot 1_{R_i} = (0,\ldots, c_i,\ldots, 0) \in R$$
i.e., $c_i \in \mathfrak q_i$.
\end{proof}

Let $v_i$ be the valuation of $\overline{R_i}$ and let
$a_i = v_i(\mathfrak q_i) = \mathrm{dim}_k \left(
\faktor{\overline{R_i}}{\mathfrak q_i \cdot
\overline{R_i}}\right)$ so that
    $$ \mathfrak C_{\widetilde{R}/R} \cdot \overline{R} =
    \left (t_1^{a_1}, \ldots, t_{\rho}^{a_{\rho}}\right)
    \cdot \overline{R} $$
and let $\mathfrak C_{\overline{R}/\widetilde{R}}$ be the
conductor of $\overline{R}/\widetilde{R}$. As
$\overline{R}$ is a principal ideal domain,
	$$ \mathfrak C_{\overline{R}/\widetilde{R}}
    = \left(t_1^{b_1}, \ldots,
	t_{\rho}^{b_{\rho}} \right) \cdot \overline{R} $$
As an ideal of $\widetilde{R}$ the $i^{th}$ component of
$\mathfrak C_{\overline{R}/\widetilde{R}}$ is given by
	$$ \left(\mathfrak C_{\overline{R}/
    \widetilde{R}}\right)_i
    = (t_i^{b_i}, t_i^{b_i+1},
	\ldots, t_i^{b_i+e_i-1}) $$
where $e_i$ is the multiplicity of $R_i$. Furthermore set
$b_i^{\prime} = \mathrm{max}\{b_i, 1\}$ and $\mathfrak C
= \left(t_1^{b_1^{\prime}}, \ldots,
t_{\rho}^{b_{\rho}^{\prime}}\right) \cdot \overline{R}$.
Note that $\mathfrak C$ is a common ideal of
$\overline{R}$ and $\widetilde{R}$ and that in fact
	$$ \mathfrak C = \mathfrak C_{\overline{R}/
    \widetilde{R}} \cap \widetilde{\mathfrak m}
    \subseteq \widetilde{\mathfrak m}, $$
where $\widetilde{\mathfrak m}$ denotes the Jacobson ideal
of $\widetilde{R}$. Furthermore, $J = \mathfrak C \cdot
\mathfrak C_{\widetilde{R}/R}$ is a common ideal of
$\widetilde{R}$ and of $\overline{R}$, and as an ideal
of $\overline{R}$ it is
	\begin{equation}\drlabel{eq:conductor-norm}
	J = \left(t_1^{a_1+b_1^{\prime}}, \ldots,
    t_{\rho}^{a_{\rho}+b_{\rho}^{\prime}} \right),
	\end{equation}
whereas as an ideal of $\widetilde{R}$, its $i^{th}$
component is given by
	\begin{equation}\drlabel{eq:conductor-tilde}
	J_i = \mathfrak q_i \cdot t_i^{b_i^{\prime}} +
    \mathfrak q_i \cdot t_i^{b_i^{\prime}+1} +
	\cdots + \mathfrak q_i \cdot t_i^{b_i^{\prime}+e_i-1}.
	\end{equation}
Note furthermore that
	$$ J \subseteq \mathfrak C_{\overline{R}/\widetilde{R}} \cdot
	\mathfrak C_{\widetilde{R}/R}
	\subseteq \mathfrak C_{\overline{R}/R}, $$
where $\mathfrak C_{\overline{R}/R}$ is the conductor
of~$\overline{R}/R$.

\begin{satz}
Let $R$ be a multibranched curve with $\rho$ branches,
let $d_R = \mathrm{max} \{ a_i + b_i^{\prime} \vert \,
i = 1, \ldots, \rho\}$, and let
$\varphi: k[[X_1, \ldots , X_n]] \longrightarrow
\overline{R}$ be a presentation of~$R$. Then $\varphi$ is
$(d_R-1)$-$\mathcal{L}$-determined.
\end{satz}

\begin{proof}
Let $\psi: k[[X_1, \ldots , X_n]] \longrightarrow
\overline{R}$ be another presentation with
	$$ \varphi = \psi \pmod{\overline{\mathfrak m}^{d_R}} =
	\left(t_1^{d_R}, \ldots, t_{\rho}^{d_R}\right) \cdot
    \overline{R} \subseteq J. $$
Thus for each $i \in \{1, \ldots, \rho\}$ and each
$j \in \{1, \ldots, n\}$ we have
	$\varphi(X_j)_i - \psi(X_j)_i \in J_i$,
i.e.,
	$$ \varphi(X_j)_i-\psi(X_j)_i=
    \sum\limits_{l=0}^{e_i-1} q_{i,l} 
    \cdot t_i^{b_i^{\prime}+l}. $$
for some $q_{i,l} \in \mathfrak q_i$. Hence if
$t_{i,l} \in \mathfrak m$ is any lift of
$t_i^{b_i^{\prime}+l}$ in $R$, then
	$$ \varphi(X_j)-\psi(X_j)= \sum\limits_{i=1}^{\rho}
    \sum\limits_{l=0}^{e_i-1}  q_{i,l}\cdot t_{i,l}. $$
Letting $Q_{i,j}, T_{i,j} \in (X_1, \ldots, X_n) \cdot
k[[X_1, \ldots, X_n]]$ be any power series with
	$$ q_{i,j}
	= Q_{i,j}(\varphi(X_1), \ldots, \varphi(X_n)), 
    \quad t_{i,j}
	= T_{i,j}(\varphi(X_1), \ldots, \varphi(X_n)), $$
we obtain that
	$$ H_j = \sum\limits_{i=1}^{\rho}
    \sum\limits_{l=0}^{e_i-1} Q_{i,j} \cdot T_{i,j} $$
is in $(X_1, \ldots, X_n)^2 \cdot k[[X_1, \ldots, X_n]]$
with
	$$ \varphi(X_j) - H_{j}(\varphi(X_1), \ldots, 
    \varphi(X_n)) = \psi(X_j) $$
and therefore
	$$ \sigma : k[[X_1, \ldots, X_n]] \longrightarrow 
    k[[X_1, \ldots, X_n]], \quad
	X_i \mapsto X_i - H_i(X_1, \ldots, X_n) $$
is an automorphism of $k[[X_1, \ldots, X_n]]$ with
$\varphi \circ \sigma = \psi$, implying the proposition.
\end{proof}

Write $\mathfrak C_{\overline{R}/R} =
\left(t_1^{c_1},\ldots,
t_{\rho}^{c_{\rho}}\right) \cdot \overline{R}$ and let
$c(R) = c_1 + \cdots + c_{\rho}$ be the conductor degree
of~$R$. Then
	$$ J \subseteq \mathfrak C_{\overline{R}/ 
    \widetilde{R}} \cdot
	\mathfrak C_{\widetilde{R}/R}
	\subseteq \mathfrak C_{\overline{R}/R}, $$
hence $c_i \leq a_i + b_i^{\prime}$ for all $i$. On the
other hand, certainly
	$$ \mathfrak C_{\overline{R}/R} \subseteq 
    \mathfrak C_{\overline{R}/\widetilde{R}},
	\quad \mathfrak C_{\overline{R}/R} \subseteq 
    \mathfrak C_{\widetilde{R}/R} \cap
	\widetilde{\mathfrak m} = J, $$
hence $\mathrm{max}\{a_i, b_i^{\prime}\} \leq c_i$. Thus
	$$ c_i\leq a_i + b_i^{\prime} \leq 2 \cdot 
    \max\{a_i, b_i^{\prime}\}
    \leq 2 \cdot c_i \leq 2 \cdot c_R $$
for all $i$, and therefore we recover Proposition 11
of~\cite{GP}:

\begin{corollary}
If $\varphi: k[[X_1, \ldots, X_n]] \longrightarrow
\overline{R}$ is a presentation of~$R$, then $\varphi$ is
$(2\cdot c(R)-1)$-$\mathcal{L}$-determined, hence
$(4 \cdot \delta(R)-1)$-$\mathcal{L}$-determined.
\end{corollary}

A better result can be obtained for plane curves. In this
case
	$$ \mathfrak C_{\overline{R}/\widetilde{R}} \cdot
	\mathfrak C_{\widetilde{R}/R} = 
    \mathfrak C_{\overline{R}/R} $$
by~\cite[17.6]{Ku}, hence $c_i = a_i + b_i$, implying
(as $\rho \geq 2$ and $a_i \geq 1$ and $b_i^{\prime}
\leq b_i + 1$ for all $i$)
	$$ c(R) = \sum\limits_{i=1}^{\rho} (a_i + b_i) \geq
    \mathrm{max}\{a_i+b_i^{\prime} \,\vert \, 
    i = 1, \ldots, \rho\} $$
and therefore we get the following corollary 
(cf.~\cite[Theorem 2.1]{Ng20}):

\begin{corollary}
If $R$ is a plane multibranched curve singularity with
$\rho$ branches ($\rho \geq 2$) and $\varphi: 
k[[X_1, \ldots, X_n]] \longrightarrow \overline{R}$ 
is a presentation of~$R$,
then $\varphi$ is $(c(R)-1)$-$\mathcal{L}$-determined,
hence $(2 \cdot \delta(R)-1)$-$\mathcal{L}$-determined.
\qed
\end{corollary}

\section{Hironaka's theorem on equivalence of 
curve singularities \label{sect:Hironaka}}

In~\cite[Theorem B]{Hi}, Hironaka showed that for any 
two reduced and irreducible curve singularities 
$(R, \mathfrak m, k)$ and $(R', \mathfrak m', k')$ over 
an algebraically closed field $k$, the existence of 
some $k$--algebra isomorphism $\varphi_{j+1} :
\faktor{R}{\mathfrak m^{j+1}} \longrightarrow
\faktor{R'}{{\mathfrak m'\,}^{i+1}}$ for $j \geq 3
\cdot \delta(R) + 1$ implies that there is an isomorphism
$\psi: \widehat{R} \longrightarrow \widehat{R'}$ of 
the completions which agrees with $\varphi$ modulo 
$\mathfrak m^{j+1-3\cdot\delta(R)}$. 

Using a different approach, Elias proved  
in~\cite[Theorem 6]{Elias}, that the existence of some 
$k$--algebra isomorphism  $\varphi_{j+1} : 
\faktor{R}{\mathfrak m^{j+1}} \longrightarrow 
\faktor{R'}{{\mathfrak m'\,}^{j+1}}$ for $j \geq 2 \cdot \mu(R)$, 
where $\mu(R)$ is the Milnor number of $R$, implies that 
there is an isomorphism
$\psi: \widehat{R} \longrightarrow \widehat{R'}$ 
of the completions. 
If $R$ is unibranch, this shows that if
$j \geq 4 \cdot \delta(R)$ 
then there is an isomorphism
$\psi: \widehat{R} \longrightarrow \widehat{R'}$ of the completions.

In this section we will strengthen these bound in case 
$R$ is unibranch and prove that $3 \cdot \delta(R)$ 
resp. $\mu(R)$ can 
be replaced by $2 \cdot \delta(R)$.

For set up,
let $(R, \mathfrak{m}, k)$ and $(R', \mathfrak{m}', k')$
be equicharacteristic complete one-dimensional local 
noetherian domains with algebraically closed residue 
fields and with local integral closures $\overline{R}$ 
and $\overline{R'}$, respectively.

We write $\overline R = k[[t]]$
for some regular parameter~$t$ of~$\overline{R}$,
and we write $v$ for the $t$-adic valuation on the 
field of fractions of~$R$. By abuse of notation,
we write $\overline{R'} = k'[[t]]$
for some regular parameter~$t$ of~$\overline{R'}$.

\begin{lemma}\drlabel{lem:e-H-mod}
Let $(R, \mathfrak{m}, k)$ and $(R', \mathfrak{m}', k')$ 
be as in the set-up. Let $d, j$ be positive integers 
such that $e(R) - 1 \le d$
and $j \geq 2 \cdot d + 1 $.
Suppose there exists a $k$-algebra isomorphism
	$\varphi: R/\mathfrak{m}^{j+1} \longrightarrow
	R'/{\mathfrak{m}'}^{j+1}$.
Then
\begin{itemize}
\item[i)] For all $i = 0, \ldots, j$,
$\varphi$ induces an isomorphism from
$\faktor{\mathfrak{m}^i}{\mathfrak{m}^{i+1}}$
to
$\faktor{\mathfrak{(m')}^i}{\mathfrak{(m')}^{i+1}}$.
In particular,
$k = k'$ up to isomorphism.
\item[ii)] $e(R) = e(R')$.
\item[iii)] $H_{R}(n) = H_{R'}(n)$ for all $n \in \N$.
\item[iv)]
If $x \in R$ generates a minimal reduction of 
$\mathfrak m$, i.e.,
if $v(x) = e(R)$,
and if for some $y \in \mathfrak{m}'$, $\varphi(x + \mathfrak{m}^{j+1}) 
= y + {\mathfrak{m}'}^{j+1}$,
then $y$ generates a minimal reduction of $\mathfrak m'$,
i.e.,
$v(y) = e(R') = e(R)$.
\end{itemize}
The condition $e(R) - 1 \leq d$ is satisfied by 
$d = \delta(R)$
(see the comment above Lemma \ref{lem:an-delta}).

\end{lemma}

\begin{proof}
The isomorphism $\varphi$ takes the units 
resp. nilpotents of 
$\faktor{R}{\mathfrak m^{j+1}}$ bijectively 
to the units resp. nilpotents 
of $\faktor{R'}{\mathfrak {m'}^{j+1}}$,
hence it takes
$\faktor{\mathfrak m}{\mathfrak m^{j+1}}$
isomorphically to 
$\faktor{\mathfrak {m'}}{\mathfrak {m'}^{j+1}}$,
and thus $\faktor{\mathfrak m^i}{\mathfrak m^{i+1}}$
isomorphically to 
$\faktor{\mathfrak {m'}^i}{\mathfrak {m'}^{i+1}}$
for all $i \le j$.
This proves i).

Let $r$ be the reduction number of~$\mathfrak m$.
Then $r \leq e(R) -1$ (cf.~\cite{Ro}),
so that $r \leq d < 2 d + 1 \le j$.

Let $x \in \mathfrak m$ generate a minimal reduction 
of~$\mathfrak m$. By~\cite[page 504]{Hu},
$(x)$ has reduction number equal to~$r$.
Consider the commutative diagram for 
$n = r, r+1, \ldots, j-1$:
	$$ \begin{tikzcd}
	\faktor{\mathfrak{m}^n}{\mathfrak{m}^{n+1}}
	\arrow[r, "\mu_x"] \arrow[d, "\varphi"] &
	\faktor{\mathfrak{m}^{n+1}}{\mathfrak{m}^{n+2}}
	\arrow[d, "\varphi"] \\
	\faktor{\mathfrak{(m')}^n}{\mathfrak{(m')}^{n+1}}
	\arrow[r, "\mu_{\varphi(\overline{x})}"] &
	\faktor{\mathfrak{(m')}^{n+1}}{\mathfrak{(m')}^{n+2}}
	\end{tikzcd} $$
where by $\overline{\phantom{x}}$ we
denote images modulo $\mathfrak{m}^{j+1}$.
By assumptions, the two vertical arrows and the top 
horizontal arrow are isomorphisms,
so by commutativity,
the bottom arrow is an isomorphism as well.
Thus by Remark~\ref{rem:ered} and since $r \le j-1$,
we have that any $y \in \mathfrak{m'}$ 
with  $(y) + {\mathfrak{m}'}^{j+1} = 
(\varphi(\overline{x}))$ 
generates a reduction of $\mathfrak{m'}$
with reduction number at most~$r$.
Furthermore,
	$$ \begin{aligned}
	e(R')
	&= H_{R'}(n)
	= H_{R'}(r)
	= \dim_k \left(\faktor{\mathfrak{(m')}^r}{ 
    \mathfrak{(m')}^{r+1}}\right) \\
	&= \dim_k \left(\faktor{\mathfrak{m}^r}{ 
    \mathfrak{m}^{r+1}}\right)
	= H_{R}(r)
	= H_{R'}(n)
	= e(R)
	\end{aligned} $$
for all $n \geq r$.
As $H_R(n) = H_{R'}(n)$ for $n < r$ via the 
isomorphism $\varphi$,
we conclude that $H_R(n) = H_{R'}(n)$ for all $n \in \N$.
This proves ii) and iii).

Finally,
$x \in \mathfrak m$ generates a minimal reduction 
of~$\mathfrak m$ if and only
if $v(x) = e(R)$.
As $y$ generates a minimal reduction 
of~$\mathfrak{m}'$, this implies that 
$v(y) = e(R') = e(R)$.
This proves iv).
\end{proof}

\begin{proposition}\drlabel{prop:conductor-mod-pow}
Let $(R, \mathfrak{m}, k)$ be as in the set-up,
let $e$ be the multiplicity of~$R$,
let $\overline{R} = k[[t]]$ be its normalization,
let $c \geq 1$
and let $I \subseteq R$ be an ideal.
Then the following are equivalent:
\begin{itemize}
\item[i)] $I = (t^c, t^{c+1}, \ldots, t^{c+e-1})$
(in particular $I \subseteq \mathfrak C_R$).
\item[ii)] There exists a chain of ideals
	$I = J_0 \supseteq J_1 \supseteq \cdots 
    \supseteq J_{e-1}$
such that
\begin{enumerate}
\item $\mathrm{dim}_k\left(\faktor{J_i}{J_{i+1}}\right) 
= 1$ for $i = 0, \ldots, e-2$.
\item $\mathrm{dim}_k\left(\faktor{J_i}{J_i^2}\right) 
= c+i$ for $i = 0, \ldots, e-1$.
\item The reduction number $r(J_i) = 1$ for 
$i = 0, \ldots, e-1$.
\end{enumerate}
\noindent
\end{itemize}

Moreover, $I = \mathfrak C_R$ is the conductor ideal 
of~$R$ if and only if $I$ satisfies any one of the 
equivalent conditions for some $c \geq 1$ and there 
exists no pair $(I', c')$ with $1 \leq c' < c$
that also satisfies these conditions.
\end{proposition}

\begin{proof}
Note that $I$ contains all $t^s$ for $s \geq c$
as $e$ is the multiplicity of~$R$,
and since all must be in $R$,
necessarily $I$ is in the conductor.
Furthermore, the ideals 
$J_i = (t^{c+i}, t^{c+i+1}, \ldots, t^{c+i+e-1})$ form 
a chain of ideals satisfying the conditions of ii)
(the reduction number is $1$ with $(t^{c+i})$ being 
a minimal reduction).

Conversely, let 
$I = J_0 \supset J_1 \supseteq \cdots 
\supseteq J_{e-1}$
be a chain as in ii).
By assumption,
$c+i = \mathrm{dim}_k\left(\faktor{J_i}{J_i^2}\right) 
= H_{J_i}(1)$,
and since the reduction number of $J_i$ is $1$,
by Remark~\ref{rem:ered},
$H_{J_i}(1) = e(J_i;R)$.
By~\cite[Theorem 11.2.7]{SH},
this equals $e(J_i\overline{R}; \overline{R}) = 
e(J_i k[[t]]; k[[t]])$,
which implies that $J_i \cdot \overline{R} 
= t^{c+i} \cdot \overline{R}$.
Hence $J_i$ contains an element of valuation $c+i$.
Therefore $I$ contains elements of valuations 
$c, c+1, \ldots, c+e-1$,
so all elements of valuation $c$ and higher,
but none of smaller valuation.
Hence (by standard approximation arguments),
$I$ contains all powers $t^j$ for $j \geq c$.
Thus i) is satisfied.

The conductor ideal $I = \mathfrak C_R$ obviously 
satisfies the conditions of ii) with $c = c(R)$, 
and clearly there cannot be a pair $(I', c')$ with 
$1 \leq c' < c(R)$ satisfying
these conditions.
This completes the proof.
\end{proof}

\begin{lemma}\drlabel{lem:phi-cond-mod-preserve}
Let $(R, \mathfrak{m}, k)$ and $(R',\mathfrak{m}', k')$ 
be as in the set-up.
Let $e(R) \geq 2$,
and $j$ be an integer such that
$j \cdot e(R) \ge 2 \cdot (c(R) + e(R) - 1)$.
Assume that there exists a $k$-algebra isomorphism
	$$ \varphi: \faktor{R}{\mathfrak{m}^{j+1}} 
    \longrightarrow
	\faktor{R'}{{\mathfrak{m}'}^{j+1}}. $$
Then the preimage in $R'$ of
$\varphi\left(
\faktor{(\mathfrak C_R + \mathfrak m^{j+1})}{ 
\mathfrak m^{j+1}}\right)$
is $\mathfrak C_{R'}$, the conductor of $R'$,
and if $x \in R$ generates a minimal reduction of 
$\mathfrak C_R$ and $x' \in R'$ 
is an element with $x'\pmod{{\mathfrak{m}'}^{j+1}} =
\varphi(x\pmod{\mathfrak{m}^{j+1}})$, then
$x'$ generates a minimal reduction of $\mathfrak C_{R'}$.

The hypotheses on $j$ are satisfied
for all $j \ge \frac {5}{e(R)} \cdot \delta(R) + 1$,
and even for all $j \ge 2 \delta(R) + 1$
when $e(R) = 2$.
So in all cases,
we may take $j \ge 2 \delta(R) + 1$.
\end{lemma}

\begin{proof}
By Lemma~\ref{lem:e-H-mod}, $e(R') = e(R)$ and $k = k'$.
Set $e = e(R)$.

Observe that for any ideals $I, J, K$ in any ring $S$,
if $I \subseteq J \subseteq K$
and $\overline{\phantom{r}}$ denotes images 
modulo $I$, then $\faktor{K}{J} = \faktor{(K+I)}{(J+I)} 
= \faktor{\overline K}{\overline J}$.

Below,
$\overline{\phantom{h}}$ denotes images 
modulo $\mathfrak{m}^{j+1}$,
$\overline{\phantom{h}}'$ images after applying $\varphi$,
and $X'$ denotes, if $X$ is an ideal of $R$,
the preimage in $R'$ of $\overline{X}'$,
and $X'$ denotes, if $X$ is an element of $R$,
a choice of a preimage in $R'$ of~$\overline{X}'$.
Note that, if $K^2$ contains $\mathfrak{m}^{j+1}$,
then $(K')^2 = (K^2)'$.
Hence for any non-negative integer $i \le j + 1$,
$(\mathfrak{m}')^i$ of $R'$ is the 
$i$th power of the maximal ideal $\mathfrak{m}'$ of $R'$,
which is the same as the preimage in $R'$ 
of~$\overline{\mathfrak{m}^i}'$.
Any containment relations in $R$
remain so in $\overline R$, $\overline R'$ and in $R'$.

Set $J_i = (t^{c(R)+i}, t^{c(R)+i+1}, \ldots, 
t^{c(R)+i+e-1})$ for $i = 0, \ldots, e-1$.
Then $\mathfrak C_R = J_0 \supseteq J_1 \supseteq \cdots 
\supseteq J_{e-1}$,
$\mathrm{dim}_k\left(\faktor{J_i}{J_{i+1}}\right) = 1$  
for each $i = 0, \ldots, e-2$,
and for each $i = 0, \ldots, e-1$,
$\mathrm{dim}_k\left(\faktor{J_i}{J_i^2}\right) = c(R)+i$ 
and the reduction number of $J_i$ is $1$
with a principal minimal reduction generated by some $y_i$
($y_i$ can be $t^{c(R)+i}$).

As $v(\mathfrak{m}^j) = j \cdot e$, the assumption 
$j \cdot e \ge 2 \cdot (c(R) + e - 1)$ implies that
	$J_{i}^2 \supseteq \mathfrak{m}^j$ for all~$i$.
Therefore, with the observation and notation 
above and the fact that $\varphi$ is an isomorphism,
we get that for all $i \le e-2$, respectively for all 
$i \le e-1$:
$$
\begin{aligned}
\mathrm{dim}_k\left(\faktor{{J_i}'}{{J_{i+1}}'}\right) &=
\mathrm{dim}_k\left(\faktor{\overline{J_i}'}{\,
\overline{J_{i+1}}'}\right) =
\mathrm{dim}_k\left(\faktor{\overline{J_i}}{\,
\overline{J_{i+1}}}\right) \\
&=\mathrm{dim}_k\left(\faktor{J_i}{J_{i+1}}\right) = 1, \\
\end{aligned} $$
and similarly
	$$ \mathrm{dim}_k\left(\faktor{{J_i'}}{ 
    {J_i'}^2}\right) =  
	\mathrm{dim}_k\left(\faktor{J_i}{J_i^2}\right) 
    = c(R)+i. $$

From $(y_i) \cdot J_i = J_i^2$
we deduce that $(\overline {y_i}) \cdot \overline{J_i} 
= \overline{J_i^2}$, $(\overline {y_i}') \cdot 
\overline{J_i}' = \overline{J_i^2}'$,
and
$$
({y_i}') \cdot {J_i'} \subseteq {(J_i^2)'} = {J_i'^2}
\subseteq
({y_i}') \cdot {J_i'} + \mathfrak{m'}^{j+1}.
$$
As $\mathfrak{m'}^{j+1} \subseteq \mathfrak m J_i'^2$,
Nakayama's Lemma proves that
$({y_i}') \cdot {J_i'} = {J_i'^2}$. Thus the 
condition ii) of 
Proposition~\ref{prop:conductor-mod-pow} is 
satisfied for $(\mathfrak{C}_R)' = J_0' 
\supseteq J_1' \supseteq J_2' 
\supseteq \cdots \supseteq J_{e-1}'$,
and therefore
	$$ (\mathfrak{C}_R)' =J'_0 = 
	(t^{c(R)}, t^{c(R)+1},\ldots, t^{c(R)+e-1}) \cdot R' 
	\subseteq k[[t]] $$
and $(\mathfrak{C}_R)'$ is contained in $\mathfrak C_{R'}$,
implying in particular that $c(R') \leq c(R)$.

Set $K_i = (t^{c(R')+i}, t^{c(R')+i+1}, 
\ldots, t^{c(R')+i+e-1}) \subseteq R'$.
Since $j \cdot e \geq 2(c(R) + e-1) \geq 2(c(R') + e-1)$
and since $e = e(R) = e(R')$,
the same reasoning applied to $\varphi^{-1}$
proves that $c(R) \leq c(R')$.

Hence $(\mathfrak C_R)' = \mathfrak C_{R'}$, and a 
minimal reduction of $\mathfrak C_{R}$
gets mapped to a minimal reduction of $\mathfrak C_{R'}$.

It remains to prove the last paragraph of the lemma.
By Remark~\ref{rem:genus-delta},
$2 \cdot \delta(R) \ge c(R)$
and $\delta(R) \geq e-1$ by a comment above 
Lemma~\ref{lem:an-delta}.
Thus $5 \cdot \delta(R) > 2 c(R) + e - 2$,
whence
$5 \cdot \delta(R)  + e > 2 \cdot (c(R) + e - 1)$.
It follows that
$\frac {5}{e} \cdot \delta(R) + 1 > 
\frac{2}{e} \cdot (c(R) + e - 1)$,
so that if $j \ge \frac {5}{e} \cdot \delta(R) + 1$,
then $j \cdot e > 2 \cdot (c(R) + e - 1)$.
When $e = 2$,
if $j \ge 2 \delta(R) + 1$,
then $j \cdot e \ge (2 \delta(R) + 1) \cdot e 
= 4 \delta(R) + 2 \ge 2 c(R) + 2
= 2 \cdot (c(R) + e - 1))$.
This finishes the proof.
\end{proof}

\begin{notiz}\drlabel{rem:cond-2delta-m}
If $j \ge 3 \cdot \frac {c(R)}{e(R)}$, 
then $\mathfrak C_R^3 \supseteq \mathfrak m^j$.
As in the proof of Lemma~\ref{lem:phi-cond-mod-preserve} 
we then have
$$
(\mathfrak m')^{j+1}
\subseteq (\mathfrak m')^j
\subseteq \mathfrak (C_R^3)'
= (\mathfrak C_R')^3
= (\mathfrak C_{R}')^3.
$$
By Remark~\ref{rem:genus-delta},
the condition on $j$ is satisfied by all 
$j \ge 6 \cdot \frac {\delta(R)}{e(R)}$, 
hence in particular for $j \geq 2 \cdot \delta(R)$ 
if $e(R) \geq 3$.
\end{notiz}

The following technical lemma will play a crucial role:

\begin{lemma}\drlabel{lem:lifting-blowup}
Let $(R, \mathfrak{m}, k)$ and $(R', \mathfrak{m}', 
k' = k)$ be as in the set-up,
let $I, J \subseteq R$ and $I', J' \subseteq R'$ be ideals,
$\widetilde{R}$ the blowup algebra of $J$ in $R$
and $\widetilde{R'}$ the blowup algebra of $J'$ in $R'$.
Assume that for some $\alpha$
there exists a $k$-algebra isomorphism
	$\varphi: \faktor{R}{\mathfrak{m}^{\alpha}} 
    \longrightarrow
	\faktor{R'}{{\mathfrak{m}'}^{\alpha}}$
and that
\begin{enumerate}
\item
$I \cdot J^3 \supseteq {\mathfrak m}^{\alpha}$ and
$I' \cdot {J'}^3 \supseteq {\mathfrak m'}^{\alpha}$.
\item $\varphi(I+\mathfrak{m}^{\alpha}) 
= I'+{\mathfrak{m}'}^{\alpha}$
and $\varphi(J+\mathfrak{m}^{\alpha}) 
= J'+{\mathfrak{m}'}^{\alpha}$.
\item $J$ and $J'$ have reduction number $1$, and if $(x)$
is a minimal reduction of $J$ and $x'$ is an element 
of $R'$ with 
$x' = \varphi(x) \pmod{{\mathfrak m'}^{\alpha}}$, 
then$(x')$ is a minimal reduction of $J'$.
\item $I \cdot J$ is a common ideal of $R$ and 
$\widetilde{R}$ and $I' \cdot J'$ is a 
common ideal of $R'$ and $\widetilde{R'}$.
\end{enumerate}
Then $\varphi$ induces an isomorphism
	$$ \widetilde{\varphi}: \faktor{\widetilde{R}}{I 
    \cdot J} \longrightarrow
	\faktor{\widetilde{R'}}{I' \cdot J'} $$
of $k$-algebras such that
	$$ \widetilde{\varphi} \vert _{\faktor{R}{I \cdot J}} 
    \,\, = \varphi \pmod{I \cdot J}. $$
\end{lemma}

\begin{proof}
As $I \cdot J^3 \supseteq \mathfrak m^{\alpha}$
and $I' \cdot J'^3 \supseteq \mathfrak m'^{\alpha}$, 
the isomorphism $\varphi$ induces an isomorphism
	$\faktor{R}{I \cdot J^3} \longrightarrow 
    \faktor{R'}{I' \cdot {J'}^3}$,
which by restriction induces an isomorphism
    $ \varphi_1: \faktor{J^2}{I \cdot J^3} \longrightarrow
	\faktor{J'^2}{I' \cdot {J'}^3}. $

Let $(x) \subseteq R$ be a minimal reduction of $J$ 
(it exists because the residue field is 
algebraically closed and hence infinite)
and let $x' \in R'$ be an element with
$x' \pmod{{\mathfrak m'}^{\alpha}} = \varphi(x)$.
Then by assumption (3), $(x)$ and $(x')$ 
are minimal reductions of $J$ and $J'$, respectively,
and by~\cite[page 504]{Hu},
their reduction numbers are~1.
Thus $x J = J^2$ and $x' J' = J'^2$.

For all integers $n \ge 1$,
$x^{-n} J^n = x^{-n} x^{n-1} J = x^{-1} J$,
implying that
$\widetilde R = \bigcup\limits_n x^{-n} J^n = 
R + x^{-1}J = x^{-1} \cdot J$,
and similarly $\widetilde{R'} = {x'}^{-1} \cdot J'$.
The $R$-submodule $I \cdot J^2$ of $\widetilde{R}$
is an ideal in $\widetilde{R}$
because $x^{-1} \cdot J \cdot (I \cdot J^2) 
= x^{-1} \cdot I \cdot x \cdot J^2
= I \cdot J^2$.
Similarly, $I' \cdot J'^2$ is an ideal in $\widetilde{R'}$.

Define
$\psi : \faktor{\widetilde R}{I \cdot J} \to 
\faktor{J^2}{I \cdot J^3}$
to be the natural composition
$$
\faktor{\widetilde R}{I \cdot J}
= \faktor{x^{-1} J}{I \cdot J}
\cong
 \faktor{x^2 x^{-1} J}{x^2 \cdot I \cdot J}
\cong
\faktor{J^2}{I \cdot J^3},
$$
where the last isomorphism is given by multiplication 
with $x^2$, and similarly define
$\psi' : \faktor{\widetilde R'}{I' \cdot J'} 
\to \faktor{J'^2}{I' \cdot J'^3}$.
So $\psi$ and $\psi'$ are $R$-module isomorphisms.

Set
$\widetilde\varphi=\psi'^{-1} \circ \varphi_1 \circ \psi :
\faktor{\widetilde R}{I \cdot J} \to 
\faktor{\widetilde R'}{I' \cdot J'}$.
Then $\widetilde\varphi$ is an $R$-module isomorphism
as it is a composition of $R$-module isomorphisms.

We prove next that $\widetilde\varphi$ is a 
$k$-algebra isomorphism. It suffices to prove that 
$\widetilde \varphi$ is multiplicative.
Let $y_1, y_2 \in \widetilde R = x^{-1} J$.
Write $y_i = x^{-1} z_i$ for some $z_i \in J$.
Using that
$\varphi_1$ is a restriction of the 
multiplicative homomorphism $\varphi$,
and with standard abuse of notation,
$$ \begin{aligned}
\widetilde \varphi(y_1 \cdot y_2)
&= \psi'^{-1} \circ \varphi_1 \circ \psi 
\left(\frac{z_1}{x} \cdot \frac{z_2}{x}\right) 
= \psi'^{-1} (\varphi_1(z_1 \cdot z_2)) \\
& = \psi'^{-1} (\varphi(z_1 \cdot z_2))  
= \psi'^{-1} \left(\varphi(z_1) \cdot \varphi(z_2)\right) 
= \frac{\varphi(z_1)}{x'} \cdot \frac{\varphi(z_2)}{x'}  \\
& = \frac{x' \varphi(z_1)}{x'^2} \cdot 
\frac{x' \varphi(z_2)}{x'^2}   
= \frac{\varphi(x) \varphi(z_1)}{x'^2} \cdot 
\frac{\varphi(x)\varphi(z_2)}{x'^2}  \\
& = \frac{\varphi(x z_1)}{x'^2} \cdot 
\frac{\varphi(x z_2)}{x'^2}  
= \psi'^{-1}(\varphi(x z_1)) \cdot 
\psi'^{-1}(\varphi(x z_2))  \\
& = \psi'^{-1}(\varphi(\psi(y_1))) \cdot 
\psi'^{-1}(\varphi(\psi(y_2))) \\
&= \widetilde \varphi(y_1) \cdot \widetilde \varphi(y_2).
\end{aligned} $$

Finally,
for any $y \in R \subseteq \widetilde R$,
	$$ \begin{aligned}
	\widetilde \varphi(y) & = \psi'^{-1} \circ \varphi_1 
    \circ \psi(y)
	= \psi'^{-1} \circ \varphi_1(y \cdot x^2)
	= \psi'^{-1} (\varphi(y) x'^2) \\
	& = \varphi(y)\pmod{I \cdot J} \end{aligned} $$
\end{proof}

\begin{lemma}\drlabel{lem:hironaka1}
Let $(R, \mathfrak m, k)$ be as in the set-up with 
$e(R) = 2$. Let $\widetilde R$ be the quadratic 
transforms of $R$ with respect to ${\mathfrak m}$
and let $\widetilde{\mathfrak m}$ be the maximal 
ideal of $\widetilde R$.
Then
\begin{enumerate}
\item
$\mathfrak m$ is two-generated;
write $\mathfrak m = (x_1, x_2)$.
We may choose $x_1$ to have $t$-adic order~$2$,
so it generates a minimal reduction of $\mathfrak m$
with reduction number~$1$;
$x_2$ may be chosen to have order $2 \delta(R) + 1$.
Then $\widetilde R = k[[x_1, \frac {x_2}{x_1}]]$
and $\widetilde{\mathfrak m} = 
(x_1, \frac {x_2}{x_1}) \widetilde R$.

\item
For every integer $i \ge 1$, $\mathfrak m^i$
is a common ideal of $R$ and $\widetilde R$.

\item
For every integer $i \ge 1$,
$\mathfrak m^{i} \subseteq \widetilde{\mathfrak m}^{i}$.
If $\delta(R) \ge 2$, then also
$\widetilde{\mathfrak m}^{i+1} \subseteq \mathfrak m^{i}$.
\end{enumerate}
\end{lemma}

\proof
Since $e(R) = 2$, we have that $\mathrm{embdim}(R) = 2$, 
and therefore 
$R = k[[x_1, x_2]] \subseteq k[[t]]$ with 
Herzog--Kunz generators $x_1, x_2$
with $v(x_1) = 2$ and $v(x_2)=2 \cdot \delta(R)+1$,
$\mathfrak m$ has reduction number $1$ with 
minimal reduction $(x_1)$,
$\mathfrak C_R = (x_1^{\delta(R)}, x_2) = 
(t^{2 \cdot \delta(R)}, t^{2 \cdot
\delta(R)+1})$,
and
$\widetilde R = R \left[\frac{\mathfrak m}{x_1} \right]
    = R \left[\frac{x_2}{x_1} \right]
    = k\left[\left[x_1, \frac {x_2}{x_1}\right]\right]$.
These facts verify~(1).

Proof of (2):
$\mathfrak m^i$ is an ideal of $R$ and $\widetilde{R}$
because for any integer $l > 0$,
$$
\frac {\mathfrak m^l}{x_1^l} \cdot \mathfrak m^i
= \frac {\mathfrak m^{l+i}}{x_1^l}
= \frac {x_1^{l+i-1} \cdot \mathfrak m}{x_1^l}
= x_1^{i-1} \mathfrak m
\subseteq \mathfrak m^i.
$$

Proof of (3):
We have $\widetilde{\mathfrak m} = (x_1, \frac {x_2}{x_1})
\widetilde R$. Clearly $\mathfrak m^{i}
\subseteq \widetilde{\mathfrak m}^{i}$.
To prove that $\widetilde{\mathfrak m}^{i+1}
\subseteq \mathfrak m^{i}$ if $\delta(R) \ge 2$,
we need to show that for all non-negative integers 
$a \le i+1$, $x_1^{i+1-a} \left(\frac {x_2}{x_1}\right)^a 
\in \mathfrak m^i$. This is true for $a = 0$ or 
$a = 1$, so we may assume that $a \ge 2$.
It suffices to prove that $\left(\frac {x_2}{x_1}\right)^a 
\in \mathfrak m^{a-1}$.
We have that
$v\left(\left(\frac {x_2}{x_1}\right)\right) = (2n-1)a 
= 2(a-1) + 2n + (2n-3)(a-1) - 1$,
where $(2n-3)(a-1) - 1 \ge 0$.
Thus we can write
$\left(\frac {x_2}{x_1}\right)^a = x_1^{a-1} c$
for some $c \in t^{2n + (2n-3)(a-1) - 1} k[[t]]$,
so for some $c \in \mathfrak C_R$.
This finishes the proof of (3).
\qed

\begin{lemma}\drlabel{lem:hironaka2}
Let $(R, \mathfrak m, k)$ and $(R',\mathfrak{m'}, k')$ 
be as in the set-up with $e(R) = 2$ and $\delta(R) > 1$.
Let $(\widetilde R, \widetilde{\mathfrak m})$
(resp. $(\widetilde{R'}, \widetilde{\mathfrak m'})$)
be the quadratic transforms of $R$ (resp. $R'$)
with respect to ${\mathfrak m}$ (resp. ${\mathfrak m'}$).
Let $j \ge 2 \delta(R) + 1$
be such that there exists a $k$-algebra isomorphism
$$
\varphi: \faktor{R}{\mathfrak{m}^{j+2}} \longrightarrow
\faktor{R'}{{\mathfrak{m}'}^{j+2}}.
$$
Then there exists a $k$-algebra isomorphism
$$
\widetilde \varphi : \faktor{\widetilde{R}}{ 
\mathfrak{m}^j} \longrightarrow
\faktor{\widetilde{R'}}{{\mathfrak{m}'}^j}
$$
that restricts to the isomorphism
$\faktor{R}{\mathfrak{m}^j} \longrightarrow 
\faktor{R'}{{\mathfrak{m}'}^j}$
induced by $\varphi$.

Furthermore,
$\widetilde \varphi$ restricts to the $k$-algebra 
isomorphism
$$
\faktor{\widetilde{R}}{\widetilde{\mathfrak{m}}^j} 
\longrightarrow
\faktor{\widetilde{R'}}{{\widetilde{\mathfrak{m}'}}^j},
$$
which agrees with the $k$-algebra isomorphism
$\faktor{R}{\mathfrak{m}^{j-1}} \longrightarrow 
\faktor{R'}{{\mathfrak{m}'}^{j-1}}$
induced by $\varphi$.
\end{lemma}

\proof
By Lemma~\ref{lem:e-H-mod},
$k' = k$ (up to isomorphism),
$e(R') = e(R) = 2$,
the preimage of $\varphi(\mathfrak m^i)$ in $R'$
is $\mathfrak m'^i$ for all $i \le j+2$,
and for any principal reduction $(x_1)$ of $\mathfrak m$,
any preimage $x_1'$ in $R'$ of $\varphi(x_1)$
generates a minimal reduction of $\mathfrak m'$.
Also,
setting $\mathfrak m = (x_1, x_2)$,
we have that $\mathfrak m' = (x_1', x_2')$,
where $x_i'$ is any preimage in $R'$ of $\varphi(x_i)$.

By Lemma~\ref{lem:hironaka1},
$\widetilde R = k[[x_1, \frac {x_2}{x_1}]]$,
$\widetilde{R'} = k[[x_1', \frac {x_2'}{x_1'}]]$,
$\mathfrak m$ and $\mathfrak m'$ have reduction number~$1$,
and for all positive integers $i$,
$\mathfrak m^i$ is a common ideal of $R$ and $\widetilde R$
and $\mathfrak m'^i$ is a common ideal of $R'$ and 
$\widetilde R'$. The powers of $\mathfrak m$ and 
$\mathfrak m'$ also have reduction number~$1$.

Rewrite $\varphi$ as
$\varphi: \faktor{R}{\mathfrak{m}^{j-1} \cdot
\mathfrak{m}^{3}}\longrightarrow \faktor{R'}
{{\mathfrak{m}'}^{j-1} \cdot {\mathfrak{m}'}^3}$.
Thus by Lemma~\ref{lem:lifting-blowup},
$\varphi$ induces a $k$-algebra isomorphism
	$$ \widetilde{\varphi}:
	\faktor{\widetilde{R}}{\mathfrak{m}^{j} } =
	\faktor{\widetilde{R}}{\mathfrak{m}^{j-1} 
    \cdot \mathfrak{m}} \longrightarrow
	\faktor{\widetilde{R'}}{{\mathfrak{m}'}^{j-1} 
    \cdot \mathfrak{m}'}
	= \faktor{\widetilde{R'}}{{\mathfrak{m}'}^{j}} $$
which restricts to $\varphi$ modulo the $j$th powers:
$\faktor{R}{\mathfrak{m}^{j}} \longrightarrow
\faktor{R'}{{\mathfrak{m}'}^{j}}$.

Since $\widetilde \varphi$ takes $x_1$ to $x_1'$ and 
$\frac {x_2}{x_1}$ to $\frac {x_2'}{x_1'}$,
it restricts to an isomorphism 
$\faktor{\widetilde{R}}{\widetilde{\mathfrak{m}}^j} 
\longrightarrow
\faktor{\widetilde{R'}}{{\widetilde{\mathfrak{m}'}}^j}$.
By Lemma~\ref{lem:hironaka1},
$\widetilde{\mathfrak{m}}^j \subseteq \mathfrak m^{j-1}$
and $\widetilde{\mathfrak{m}'}^j \subseteq 
\mathfrak m'^{j-1}$,
so that the last restriction restricts to $\varphi$ 
modulo $\mathfrak m^{j-1}$.
\qed

\medskip

Hironaka, in [Hi], top of page 156, provides a
brief outline of an argument to prove the next theorem for
$3 \delta(R) + 1$ instead of $2 \delta(R) +1$.

\begin{satz}[Hironaka's equivalence of singularities]\drlabel{thm:hironaka}
Let $(R, \mathfrak{m}, k)$ and $(R', \mathfrak{m}', k')$ be
complete reduced and irreducible domains
with algebraically closed residue fields $k$, $k'$.
Assume that for some $j \geq 2 \delta(R) + 1$
there exists a $k$-algebra isomorphism
	$$ \varphi: \faktor{R}{\mathfrak{m}^{j+1}}
    \longrightarrow
	\faktor{R'}{{\mathfrak{m}'}^{j+1}}. $$
Then there exists a $k$-algebra-isomorphism
	$$ \psi: R \longrightarrow R' $$
such that
	$$ \psi \cong \varphi : \faktor{R}{\mathfrak{m}^{j+1- 2 \delta(R)}}
	\longrightarrow \faktor{R'}{{\mathfrak{m}'}^{j+1 - 2 \delta(R)}}. $$
\end{satz}

\begin{proof}
By Lemma~\ref{lem:e-H-mod},
$k' = k$ (up to isomorphism),
$e(R') = e(R)$,
and the preimage of $\varphi(\mathfrak m^i)$ in $R'$
is $\mathfrak m'^i$ for all $i$.

\medbreak
We first assume that $e = e(R) \geq 3$.
In this situation, we can follow Hironaka's strategy quite closely.

By Lemma~\ref{lem:phi-cond-mod-preserve},
the preimage of $\varphi(\mathfrak C_R)$ in $R'$ is 
$\mathfrak C_{R'}$.
By Remark~\ref{rem:cond-2delta-m},
we have
$\mathfrak{m}^{2 \delta(R)} \subseteq {\mathfrak C_R}^3$
and
${\mathfrak{m}'}^{2 \delta(R)} \subseteq 
{\mathfrak C_{R'}}^3$.
Hence with $\beta = j - 2 \delta (R) + 1$ and $\alpha = j+1$,
$\mathfrak m^{\beta} \cdot \mathfrak C^3 \supseteq \mathfrak{m}^{\alpha}$
and
${\mathfrak m'}^{\beta} \cdot {\mathfrak C'}^3 
\supseteq {\mathfrak{m}'}^{\alpha}$.
Thus $\varphi$ induces an isomorphism
    $$\varphi: \faktor{R}{\mathfrak m^{\beta}
    \cdot \mathfrak C_R^3} \longrightarrow
    \faktor{R'}{{\mathfrak m'}^{\beta} \cdot
    {\mathfrak C^3_{R'}}}.$$
As the reduction number of the conductor is always~$1$,
the assumptions of Lemma~\ref{lem:lifting-blowup} 
are satisfied with $\alpha = j+1$, 
$I = \mathfrak m^\beta$ and $J = \mathfrak C_R$.
Thus with $\widetilde R$ (resp. $\widetilde{R'}$)
being the blowup of $\mathfrak C_R$ in $R$
(resp. of $\mathfrak C_{R'}$ in $R'$),
we get an isomorphism
	$$ \widetilde{\varphi}:
	\faktor{\widetilde{R}}{\mathfrak m^{\beta} 
    \cdot \mathfrak C_R} \longrightarrow
	\faktor{\widetilde{R'}}{{\mathfrak m'}^{\beta} \cdot
    {\mathfrak C_{R'}}}, $$
which restricts
on
	$\faktor{{R}}{\mathfrak m^{\beta} \cdot \mathfrak C_R}
	\longrightarrow
	\faktor{{R'}}{{\mathfrak m'}^{\beta} \cdot 
    {\mathfrak C_{R'}}}$
to the same isomorphism as $\varphi$.

But $\widetilde{R} = y^{-1} \cdot \mathfrak C_R =
\overline{R} = k[[t]]$, and $\widetilde{R} =
{y'}^{-1} \cdot \mathfrak C_{R'}
= \overline{R'} = k[[t]]$ for some $y \in \mathfrak C_R$
and $y' \in \mathfrak C_{R'}$.
Let $\widetilde \varphi(t + \mathfrak m^\beta 
\cdot \mathfrak C_R^3)
= z + {\mathfrak m'}^{\beta} \cdot {\mathfrak C_{R'}}$
for some $z \in k[[t]]$. Necessarily 
$z \in (t) \cdot k[[t]] \setminus (t^2) \cdot k[[t]]$.
Let
	$$ \Psi: k[[t]] \longrightarrow k[[t]] $$
be the $k$-algebra isomorphism defined by
$\Psi(t) = z$. Then $\Psi$ lifts $\widetilde{\varphi}$.

Let $x_1, \ldots, x_n$ be a set of Herzog--Kunz generators 
of $R$, let $y_i = \Psi(x_i)$,
and let $x_i' \in R'$ satisfy
$x_i' \pmod{{\mathfrak m'}^{j+1}} = 
\varphi(x_i + \mathfrak m^{j+1})$.
Since $\varphi$ restricts to $\mathfrak m$
as an isomorphism onto $\mathfrak m'$,
by Nakayama's lemma, $x_1', \ldots, x_n'$
is a minimal set of generators of $\mathfrak m'$.
The diagram
	$$ \begin{tikzcd}
	\faktor{R}{\mathfrak m^{\beta} \cdot \mathfrak C_R}
	\arrow[r, "\varphi"] \arrow[d, "\iota"] &
	\faktor{R'}{{\mathfrak m'}^{\beta}\cdot {\mathfrak C_{R'}}}
	\arrow[d, "\iota' "] \\
	\faktor{k[[t]]}{\mathfrak m^{\beta} \cdot 
    \mathfrak C_R}
	\arrow[r, "\widetilde{\varphi}"] &
	\faktor{k[[t]]}{{\mathfrak m'}^{\beta}
    \cdot {\mathfrak C_{R'}}}
	\end{tikzcd} $$
commutes, and as the vertical maps are injective,
we conclude that
	$$ y_i = x_i'
	\pmod {{\mathfrak m'}^{\beta} \cdot 
    \mathfrak C_{R'} \cdot k[[t]]}. $$
Since $x_i' \in R'$ and $y_i - x_i' \in 
\mathfrak C_{R'} k[[t]]$,
it follows that $y_i \in R'$.
Also,
since ${\mathfrak m'}^{\beta} \cdot \mathfrak C_{R'}
\subseteq \mathfrak m'^\beta R' \subseteq \mathfrak m'^2$,
by Nakayama's lemma,
$y_1, \ldots, y_n \in R'$ form a set of minimal generators of $\mathfrak m'$.
Hence $\Psi$ induces the desired isomorphism
	$$ \psi: R \longrightarrow R', $$
which modulo $\mathfrak m^{\beta} \cdot \mathfrak C_{R}$,
hence modulo $\mathfrak m^{\beta}$,
agrees with the given $\varphi$.

\medbreak

Next, we treat the case $e(R) = 2$.

By Lemma~\ref{lem:hironaka1}~(1),
$R = k[[x_1, x_2]] \subseteq k[[t]]$ with $v(x_1) = 2$
and $v(x_2)=2 \cdot \delta(R)+1$,
and $(x_1)$ generates a minimal reduction of $\mathfrak m$
with reduction number~$1$.

If $\widetilde R$ is the blowup of $\mathfrak m$ in $R$,
then $\widetilde R = R[\frac {\mathfrak m}{x_1}] = 
R[\frac {x_2}{x_1}] =
k[[x_1, \frac {x_2}{x_1}]]$.
Observe that $v\left(\frac {x_2}{x_1}\right) = 
2 \cdot (\delta(R) - 1) + 1$,
so that $\delta(\widetilde R) = \delta(R) - 1$.

By Lemma~\ref{lem:e-H-mod},
$e(R') = e(R) = 2$ and any preimage $x_1'$ in $R'$ of
$\varphi(x_1)$ generates a minimal reduction of 
$\mathfrak m'$,
similarly with reduction number~$1$.
Let $x_2'$ in $R'$ be a preimage of $\varphi(x_2)$.

All powers of $\mathfrak m$ and $\mathfrak m'$ 
have reduction number~$1$.

We proceed by induction on $\delta(R)$.

\medbreak

First, assume that $\delta(R) = 1$ (and $e(R) = 2$).
Then $\widetilde{R} = \overline{R} = k[[t]]$
and $j \geq 2 \delta(R) + 1 = 3$.
By Lemma~\ref{lem:hironaka1},
$\mathfrak m^{j-1}$ is a common ideal of $R$ and 
$\widetilde{R}$,
and $\mathfrak m'^{j-1}$ is a common ideal of $R'$ and 
$\widetilde{R'}$.
By Lemma~\ref{lem:lifting-blowup},
the isomorphism $\varphi: 
\faktor{R}{\mathfrak{m}^{j-2} \cdot \mathfrak{m}^{3}}
	\longrightarrow \faktor{R'}{{\mathfrak{m}'}^{j-2}
    \cdot {\mathfrak{m}'}^3}$
induces an isomorphism
	$$ \widetilde{\varphi}:
	\faktor{\widetilde{R}}{\mathfrak{m}^{j-1} } =
	\faktor{\widetilde{R}}{\mathfrak{m}^{j-2} 
    \cdot \mathfrak{m}}
    \longrightarrow
	\faktor{\widetilde{R'}}{{\mathfrak{m}'}^{j-2} \cdot \mathfrak{m}'}
	= \faktor{\widetilde{R'}}{{\mathfrak{m}'}^{j-1}}, $$
which restricts to $\varphi$ modulo ${\mathfrak{m}}^{j-1}$
and ${\mathfrak{m}'}^{j-1}$.
Since $\widetilde R = \overline R = k[[t]]$ is regular 
and $j-1 \geq 2$, the isomorphism $\widetilde \varphi$
implies that $\widetilde{R'}$ is regular as well.
Since $R' \subseteq \widetilde {R'}$,
we have that
$k[[t]] = \overline{R'} \subseteq 
\overline{\widetilde {R'}} = \widetilde{R'}$,
and since they are in the same field of fractions,
it follows that $k[[t]] = \overline{R'} = \widetilde{R'}$.
Thus, since $j-1 \ge 2$,
by Nakayama's lemma,
$\widetilde{\varphi}$ lifts
to an isomorphism
$$
\Psi: \overline{R} \longrightarrow \overline{R'}.
$$
Setting $y_i = \Psi(x_i)$ ($i = 1,2$),
we have that $y_i = x_i' \pmod{{\mathfrak{m}'}^{j-1}}$,
and as ${\mathfrak{m}'}^{j-1}$ is a common ideal of 
$R'$ and $\overline{R'}$, $y_i \in R'$,
and as $j-1 \geq 2$,
by Nakayama's lemma,
$y_1$, $y_2$ generate $R'$ over $k$,
implying that $\Psi$ induces the desired isomorphism 
$\psi: R \longrightarrow R'$ in this
case.

\medbreak

Now assume that $\delta(R) > 1$ and still 
$e(R) = 2 = e(R')$.
In this case, we can write 
$R = \faktor{k[[X,Y]]}{(f(X,Y))}$ and
$R' = \faktor{k[[X,Y]]}{(g(X,Y))}$
for some $f(X,Y), g(X,Y) \in k[[X,Y]]$ of total order $2$.
By possibly changing the generators of~$R'$,
we may assume that $\varphi$ takes $X$ to $X$ 
and $Y$ to $Y$ modulo $(X,Y)^{j+1}$.
Setting $\mathfrak M = (X,Y) \subseteq k[[X,Y]]$
and viewing $\varphi$ as an isomorphism
	$$ \varphi: \faktor{k[[X,Y]]}{(f(X,Y))+
    \mathfrak M^{j+1}}
	\longrightarrow
	\faktor{k[[X,Y]]}{(g(X,Y))+\mathfrak M^{j+1}} $$
we get that
$f(X,Y) = U \cdot g(X,Y) - r(X,Y)$
for some $U \in k[[X,Y]]$ and some $r(X,Y) \in \mathfrak M^{j+1}$.
Since the degree of $f(X,Y)$ and $g(X,Y)$ is $2 < j+1$,
necessarily $U$ is a unit.
	After replacing $g(X,Y)$ by $U \cdot g(X,Y)$
	which does not change the ideal,
	we may assume that
		$$ f(X,Y) = g(X,Y) - r(X,Y) $$
	for some $r(X,Y) \in \mathfrak M^{j+1}$
	and $\varphi$ is viewed as $\mathrm{id}$.

	In all subcases below we construct elements
	$a, b \in (X,Y)^{j+1-2\delta(R)}$ and
	a unit $u \in k[[X,Y]]$
	such that
		\begin{equation}\label{eq:transform}
		f(X+a,Y+b) = u \cdot g(X,Y) + s(X,Y) \tag{$*$}
		\end{equation}
	for some $s(X,Y) \in (X,Y)^{j+2}$.
	Hence the base change isomorphism $\pi$,
	given by $\pi(X) = X+a$, $\pi(Y) = Y+b$,
	followed by $\varphi$,
	induces an isomorphism
		$$ \begin{aligned}
		\varphi^{*}: \faktor{k[[X,Y]]}{(f(X,Y)) + \mathfrak M^{j+2}}
		 \longrightarrow \, \, &
		\faktor{k[[X,Y]]}{(u \cdot g(X,Y)) + \mathfrak M^{j+2}} \\
		& = \faktor{k[[X,Y]]}{(g(X,Y)) + \mathfrak M^{j+2}}
		\end{aligned} $$
	of $k$-algebras
	with $\varphi^{*}\pmod{\mathfrak M^{j+1-2\delta(R)}} = \mathrm{id}$,
	hence also an isomorphism
	$\varphi^{*}: \faktor{R}{\mathfrak m^{j+2}}
	\longrightarrow \faktor{R'}{{\mathfrak m'}^{j+2}}$
	with
	$$ \begin{aligned}
		\varphi^{*} \pmod{\mathfrak m^{j+1-2\delta(R)}}
		= \varphi \pmod{\mathfrak m^{j+1-2\delta(R)}}.
		\end{aligned} $$
    Thus, by Lemma~\ref{lem:hironaka2}, $\varphi^{*}$ 
    extends to an isomorphism
		$$ \widetilde{\varphi^{*}}:
		\faktor{\widetilde{R}}{\widetilde{\mathfrak m}^{j}}
		\longrightarrow
		\faktor{\widetilde{R'}}{\widetilde{\mathfrak m'}^{j}} .$$
	where $(\widetilde{R}, \widetilde{\mathfrak m})$ resp.
	$(\widetilde{R'}, \widetilde{\mathfrak m'})$ are the quadratic
	transforms of $(R, \mathfrak m)$ resp. $(R', \mathfrak m')$.

	As $\delta(\widetilde{R}) = \delta(R) -1$, by induction the isomorphism
	$\widetilde{\varphi^{*}}$ induces an isomorphism
	$\Psi: \widetilde{R} \longrightarrow \widetilde{R'}$ 
    such that
	    $$ \Psi = \widetilde{\varphi^{*}}:
		\faktor{\widetilde{R}}{\widetilde{\mathfrak m}^{j-2 \cdot
	    (\delta(R)-1)}}
		\longrightarrow
		\faktor{\widetilde{R'}}{ 
        \widetilde{\mathfrak m'}^{j-2 \cdot
	    (\delta(R)-1)}}. $$
	By Lemma~\ref{lem:hironaka2},
	$\Psi$, $\widetilde{\varphi^{*}}$ and $\varphi$
	restrict to the same isomorphism 
		$$ \faktor{R}{\mathfrak m^{j+1-2 \delta(R)}}
	    \longrightarrow
		\faktor{R'}{{\mathfrak m'}^{j+1-2 \delta(R)}}. $$
	This implies that
	    $$ \begin{aligned}
	    \Psi(X) \cong \varphi(x) \cong X
	    \!\!\pmod{(g(X,Y)) + (X,Y)^{j+1-2 \delta(R)}}, \\
	    \Psi(Y) \cong \varphi(y) \cong Y
	    \!\!\pmod{(g(X,Y)) + (X,Y)^{j+1-2 \delta(R)}},
	    \end{aligned} $$
	so that $\Psi(X), \Psi(Y) \in \mathfrak m' 
    \subseteq R'$ generate $\mathfrak m'$,
	hence they generate $R'$ as an analytic $k$-algebra.
	Therefore, 
	$\Psi$ restricts to an isomorphism
		$$ \psi : R \longrightarrow R' $$
	with the desired properties
	(as $\mathfrak m^{j+1-2 \delta(R)}$ is a common ideal
	of $R$ and $\widetilde{R}$ and
	${\mathfrak m'}^{j+1-2 \delta(R)}$ is a common ideal
	of $R'$ and $\widetilde{R'}$).

	\smallskip

	Thus, in all cases below it suffices to get to the
	conditions in ($*$)
	in the set-up of $e(R) = 2$
	and $f(X,Y) = g(X,Y) \textcolor{brown}{-} r(X,Y)$ for some $r(X,Y) \in (X,Y)^{j+1}$.
	We shorten $\delta(R)$ to $\delta$.

	First, suppose that $\mathrm{char}(k) \neq 2$.
	We may assume that $x_1 = t^2 \cdot \alpha$
	where $\alpha = 1 + \alpha_1 \cdot t + \cdots $.
	As over the ring $\left(\faktor{\overline{R}}{
	\overline{\mathfrak m}}\right)[T] = k[T]$,
	the polynomial $T^2 - \overline{\alpha} = T^2 - 1$ is
	separable with $1$ being one root, by Hensel's Lemma
	(\cite[Theorem 7.3]{Ei}, see also
	\cite[proof of Theorem 3.1]{HMM}),
	we may change the parameter~$t$ of $\overline{R}$
	in such a way that $x_1 = t^{2}$.
	As $c(R) = 2 \delta < v(x_2) = 2 \delta +1$,
	we may replace $x_2$ by $t^{2\delta+1}$ by~\cite[Corollary 3.3]{HHMM}.
	Thus $R = k[[t^2, t^{2\delta+1}]]$
	and $f(X,Y) = Y^2-X^{2\delta+1}$
	(cf. also~\cite[Chapter V, \S 1 for $k = \C$]{Za}).

	Suppose in addition to $\mathrm{char}(k) \neq 2$
	that also $\mathrm{char}(k) \not = 2\delta+1$.
	As $j \geq 2\delta+1$, we may write $r(X,Y)$ from the set-up as
		$$ r(X,Y) = X^{2\delta} \cdot A +Y \cdot B $$
	with $A \in (X, Y)^{j+1-2\delta}$ and 
    $B \in (X,Y)^{j}$.
	Set $a = - \frac {A}{2 \delta + 1}$, 
    $b = \frac {B}{2}$,
	$u = 1$. Then
    \allowdisplaybreaks
	$$ \begin{aligned}
	f\left(X + a, Y + b\right)
	&=
	\left(Y + \frac {1}{2} \cdot B\right)^2
	- \left(X - \frac {1}{2\delta+1} 
    \cdot A\right)^{2\delta+1} \\
	&=
	Y^2 + Y \cdot B + \frac {1}{4} \cdot B^2 - X^{2\delta+1} + X^{2\delta} \cdot A + C \\
	& = Y^2 - X^{2\delta+1} + r(X,Y) + 
    \frac {1}{4} \cdot B^2 + C \\
	& = g(X,Y) + \frac {1}{4} \cdot B^2 + C
	\end{aligned} $$
	with $C \in (X,Y)^{2\delta-1+2j+2-4\delta} =
	(X,Y)^{2j+1-2\delta} \subseteq (X,Y)^{j+2}$ and
	$B^2 \in (X,Y)^{2j} \subseteq (X,Y)^{j+2}$.
	Hence conditions ($*$) are satisfied,
	so the theorem is proved in this case.

	\smallbreak

	Now suppose that $\mathrm{char}(k) = 2\delta+1$.
	As $j \geq 2\delta+1$,
	we may write $r(X,Y)$ from the set-up as
		$$ r(X,Y) = A \cdot Y + B \cdot X^{j+1} + C $$
	for some $A \in (X,Y)^{j}$, $B \in k$ and $C \in (X, Y)^{j+2}$.
	Since the characteristic of $k$ is not~$2$,
	by Hensel's Lemma there exists 
    a power series $L \in X \cdot k[[X]]$ such that
	$$
	(1 + L)^2 = 1 + B \cdot X^{j - 2 \delta}.
	$$
	If $B = 0$, then $L = 0$,
	otherwise,
	by expansion of the square we have 
    $1 + L \cdot (2 + L) = 1 + B \cdot
	X^{j - 2 \delta}$,
	so that necessarily the order of $L$ is $j - 2 \delta$.
	Set $a = 0$, $b = \frac{1}{2} \cdot A + Y \cdot L 
    + \frac{1}{2} \cdot A \cdot
	L \in (X,Y)^{j+1 - 2\delta}$,
	$u = 1 + B \cdot X^{j-2\delta}$.
	Then
		$$ \begin{aligned}
		f(X+&a, Y + b) =
		f\left(X, \left( Y + \frac{1}{2} \cdot A\right)\cdot (1 + L) \right) \\
		&=
		\left( Y + \frac{1}{2} \cdot A\right)^2 \cdot 
        (1 + L)^2 - X^{2\delta+1} \\
		&=
		\left( Y^2 + A \cdot Y + \frac{1}{4} 
        \cdot A^2\right)
		\cdot (1 + B \cdot X^{j - 2 \delta}) \\
		&\hskip 3em
		- X^{2\delta+1} \cdot 
        (1 + B \cdot X^{j - 2 \delta})
	\\
		&\hskip 3em
		+ B \cdot X^{j+1} \cdot 
        (1 + B \cdot X^{j - 2 \delta})
		- B^2 \cdot X^{2j+1-2\delta}
	\\
		&=
		\left( Y^2 - X^{2\delta+1} + A \cdot Y + 
        B \cdot X^{j+1} \right) 
		\cdot (1 + B \cdot X^{j - 2 \delta}) \\
	&\hskip 3em
		+ \frac{1}{4} \cdot A^2 \cdot (1 + 
        B \cdot X^{j - 2 \delta})
		- B^2 \cdot X^{2j+1-2\delta},
		\end{aligned} $$
and since $A^2, X^{2j+1-2\delta} \in (X,Y)^{j+2}$,
the last expression is of the form
$u \cdot \left( f(X,Y) + r(X,Y)\right) 
= u \cdot g(X,Y)$ modulo $(X,Y)^{j+2}$.
	Thus, conditions ($*$) are satisfied,
	and the theorem is also proved in this case.

\smallbreak

It remains to prove the theorem for $e(R) = 2$
in case $\mathrm{char}(k) = 2$.
By~\cite[Proposition 3.4]{Ng20},
any such curve is
parameter--equivalent, and thus by ~\cite[Proposition
1.2.10]{Ng13} contact--equivalent, to a curve defined 
by $f(X,Y) = Y^2+ X^{2\delta+1}$ or
by $f(X,Y) = Y^2 + Y \cdot X^{i} + X^{2\delta+1}$
for some $\delta < i \leq 2 \delta$
(resp. apply~\cite[page 566]{KS}).

First, we deal with the case $f(X,Y) = Y^2+X^{2\delta+1}$ and write
	$$ r(X,Y) = a X^{2\delta} + B Y^2 + C $$
for some $a$ homogeneous of degree $j+1-2\delta$,
$B$ homogeneous of degree $j-1$ and
$C \in (X,Y)^{j+2}$.
(Note that $aX^{2\delta} + B Y^2$ represents an 
arbitrary homogeneous element of $(X,Y)^{j+1}$ of 
degree $j+1$ as $j \ge 2\delta + 1$.)
We have $(X + a)^{2\delta+1} = X^{2\delta+1} 
+ a X^{2\delta} + a'$ for some 
$a' \in (X,Y)^{2(j+1-2\delta) + 2\delta-1}
\subseteq (X,Y)^{j+2}$
as $j \ge 2\delta+1$,
and that $u = 1+B$ is a unit in $k[[X,Y]]$.
With $b = 0$
and $a'' = B \cdot (X^{2 \delta + 1} + r(X,Y) + C + a') \in (X,Y)^{j+2}$
we have
$$
\begin{aligned}
f(X+a, Y+b)
&= Y^2 + (X + a)^{2\delta+1} \\
&= Y^2 + X^{2\delta+1} + a X^{2\delta} + a' \\
&= (1 + B) \cdot Y^2 + X^{2\delta+1} + a X^{2\delta} 
+ B Y^2 + C + C + a' \\
&= (1 + B) \cdot Y^2 + X^{2\delta+1} + r(X,Y) + C + a' \\
&= (1 + B) \left(Y^2 + X^{2\delta+1} + r(X,Y) + C + a'\right)
+ a'' \\
&=
u \cdot g(X,Y) + D,
\end{aligned}
$$
where $D = (1+B)(C+a') + a'' \in (X,Y)^{j+2}$.
Thus, conditions ($*$) are satisfied,
and the theorem is also proved in this case.

\smallbreak

Finally, we deal with the remaining case of 
$\mathrm{char}(k) = 2$ with
    $$f(X,Y) = Y^2 + Y \cdot X^{i} + X^{2\delta+1} $$
and $\delta < i \le 2 \delta$,
by carefully evaluating the techniques of~\cite{KS}.
Write
	$$ r(X,Y) = A Y X^j + B X^{j+1} + C Y^2 + D $$
for some $A, B \in k$,
some homogeneous $C \in (X,Y)^{j-1}$ of degree $j-1$
and some $D \in (X,Y)^{j+2}$.
(Note that $A Y X^j + B X^{j+1} + C Y^2$
represents an arbitrary homogeneous element of $(X,Y)^{j+1}$
of degree $j+1$.)
Observe that $1 + A X^{j-i}$ is a unit in $k[[X,Y]]$
whose inverse $v$ satisfies
$v - 1 \in \left(X^{j-i}\right) \subseteq \left(X^{j-2 \delta}\right)$
as $i \leq 2 \delta$.
Let $B' = A+B$ if $i = 2 \delta$ and $B' = B$ otherwise.
Then $1 + B' X^{j-2 \delta}$ is a unit in $k[[X,Y]]$,
so since $2 (\delta-i) + 1$ is odd,
there exists a unit $w \in k[[X,Y]]$ such that $w^{2(\delta-i)+1}
= 1 + B' X^{j-2 \delta}$ and $w-1 \in  \left(X^{j-2 \delta}\right)$.
Let $z$ be the inverse of $1 + C$.
Set
    	$$ \begin{array} {l c l}
    	a & = & (v-1) \cdot X + (w - 1) \cdot v \cdot X
	= \left( v \cdot w - 1\right) \cdot X, \\
    	b & = & \left(v^{i+1} - 1\right) \cdot Y +
	\left(w^i - 1\right) \cdot v^{i+1} \cdot Y
	= \left(v^{i+1} \cdot w^{i} - 1 \right) \cdot Y, \\
    	u & =  & v^{2i+2} \cdot w^{2i} \cdot z,
    	\end{array} $$
with $a, b \in (X,Y)^{j+1 - 2 \delta}$ and $u$ a unit.
Then
	$$ \begin{aligned}
	v^{2 (\delta -i) - 1} w^{2 (\delta-i) + 1} X^{2 \delta + 1}
	& = \left( 1 + A X^{j-i} \right)^{2(i-\delta)+1}
	\left( 1 + B' X^{j-2 \delta}\right) X^{2 \delta + 1} \\
	& = X^{2 \delta + 1} + A X^{j-i+2 \delta + 1} + B' X^{j+1} + d
	\end{aligned} $$
for some $d \in (X,Y)^{(j-i)+(j-2\delta) + (2\delta + 1)}
= (X,Y)^{j-i+j + 1} \subseteq (X,Y)^{j+2}$.
If $i = 2 \delta$,
the above is $X^{2 \delta + 1} + B X^{j+1} + d$,
and otherwise the above equals
$X^{2 \delta + 1} + B X^{j+1} + d'$,
for some $d' \in (X,Y)^{j+2}$.
In all cases,
$$
v^{2 (\delta -i) - 1} w^{2 (\delta-i) + 1} X^{2 \delta + 1}
= X^{2 \delta + 1} + B X^{j+1} + d''
$$
for some $d'' \in (X,Y)^{j+2}$.
As $i > \delta > 1$ we have
	$$
	d''' = C \left(Y X^i + A Y X^j + X^{2 \delta + 1} + B X^{j+1} + d''\right) + d''
	\in (X,Y)^{j+2}.
	$$
Then
	$$
	\begin{aligned}
	f(&X+a, Y+b)
	= f(X + (vw - 1) \cdot X, Y+ (v^{i+1} w^{i} - 1) \cdot Y) \\
	&= f(vwX, v^{i+1} w^i Y) \\
	&= v^{2i+2} w^{2i} Y^2 + v^{i+1} w^i Y v^i w^i X^i
	+ v^{2\delta +1} w^{2\delta+1} X^{2 \delta + 1} \\
	&= v^{2i+2} w^{2i} \left(Y^2 + v^{-1} Y X^i
	+ v^{2\delta +1- 2i - 2} w^{2\delta+1- 2i} X^{2 \delta + 1} \right)\\
	&= v^{2i+2} w^{2i} \left(Y^2 + (1 + A X^{j-i}) Y X^i
	+ X^{2 \delta + 1} + B X^{j+1} + d''
	\right)\\
	&= v^{2i+2} w^{2i} z(1 + C)\left(Y^2 + Y X^i + A Y X^j
	+ X^{2 \delta + 1} + B X^{j+1} + d''
	\right) \\
	&= u \left(Y^2 + C Y^2 + Y X^i + A Y X^j
	+ X^{2 \delta + 1} + B X^{j+1} + d'''
	\right) \\
	&= u \left(f(X,Y) + r(X,Y) + D + d''' \right) \\
	&= u \cdot g(X,Y) + u(D+d''').
	\end{aligned}
	$$
Thus, conditions ($*$) are satisfied,
and the theorem is also proved in this final case.
\end{proof}

\begin{notiz}
As a general bound in $\delta(R)$, the bound
$2 \cdot \delta(R)+1$ in theorem~\ref{thm:hironaka} is
sharp as the curves given by $k[[t^2, t^{2n+1}]]$ show.

Involving more invariants, better bounds may be obtained.
Evaluating the above arguments carefully gives (in any
characteristic and for any multiplicity) that the bound
$\frac {6 \cdot \delta(R)}{e(R)} +1$ works. For curves
of multiplicity $e(R) = 2$ this recovers Hironaka's
original bound.
\end{notiz}

\begin{notiz}
The assumption that $k$ be algebraically closed is only needed
for the case $e =2$ and $\mathrm{char}(k) = 2$ as this is
required for the classification in~\cite{Ng20}
(resp.~\cite{KS}). In all other cases it would be
sufficient to assume that the residue class field is
infinite (to guarantee the existence of minimal reductions)
and that $\faktor{R}{\mathfrak m} =
\faktor{\overline{R}}{ \overline{\mathfrak m}}$ and
$\faktor{R'}{\mathfrak m'} = \faktor{\overline{R'}}{
\overline{\mathfrak m'}}$. The description of
algebroid curves of multiplicity 2 as
$R = k[[t^2,t^{2n+1}]]$ in any characteristic other
than $2$ only needs the existence of two distinct
square roots of $1$ (by Hensel's Lemma).
\end{notiz}

\bibliographystyle{plain}

\end{document}